\documentclass[12pt]{article}
\usepackage{amsmath}
\usepackage{amsthm}
\usepackage{amssymb}
\usepackage{amscd}
\usepackage[dvips]{graphicx}
\newtheorem{definition}{Definition}

\newtheorem{lemma}[definition]{Lemma}
\newtheorem{cor}[definition]{Corollary}

\newtheorem{remark}[definition]{Remark}

\newtheorem{observation}[definition]{Observation}
\newcommand{\qsi}{$q$-SI $\sigma$-differential }
\newcommand{\com}{\mathbb {C}} 
\newcommand{\Id}{\mathrm{Id}}
\newcommand{\eL}{\mathcal{L}}
\DeclareSymbolFont{largesymbol}{OMX}{yhex}{m}{n}
\DeclareMathAccent{\Widehat}{\mathord}{largesymbol}{"62}
\newcommand{\itqsi}{{\it qsi }}
\newcommand{\GL}{\mathrm GL}
\newcommand{\N}{\mathbb{N}}
\newcommand{\G}{\mathbb{G}}
\begin{document}
\title{Toward quantized Picard-Vessiot theory,  \\
Further observations on our previous example }
\author{Katsunori Saito and 
Hiroshi Umemura \\  
Graduate School of Mathematics \\
Nagoya University\\ \ 
\\
Email\quad {\small 
m07026e@math.nagoya-u.ac.jp and 
 umemura@math.nagoya-u.ac.jp} }
\maketitle
\begin{abstract}
It is quite natural to wonder whether there is a difference-differential equations, the Galois group of which is a quantum group that is neither commutative nor co-commutative.  
Believing that there was no such linear equations, 
 we explored non-linear equations and discovered a such equation \cite{saiume13.1}, \cite{saiume13.2}.
  We show that the example is related with a linear equation. We treat only one 
  charming 
  example. It is not, however, an isolated example.  
We  open a window that allows us to have a look at    
a quantized Picard-Vessiot theory. 
\end{abstract}


\section{Introduction}
We believed for a long time that it was  impossible to 
quantize Picard-Vessiot theory, 
Galois theory for 
linear difference or differential equations. 
Namely, there  was no Galois theory for linear difference-differential equations, the Galois group of which is a quantum group that is, in general,  
neither commutative nor co-commutative. 
Our mistake came from a misunderstanding of preceding   
works Hardouin \cite{har10} and Masuoka and Yanagawa \cite{masy}. 
They studied linear \qsi  equations, \itqsi equations for short, 
 under two assumptions 
on \itqsi base field $K$  and  
\itqsi module $M$: 
\begin{enumerate} 
  \item 
  The base field $K$contains $\com (t)$. 
  \item 
  On the $K [\sigma , \, \theta^* ]$-module $M$ the equality 
\[
\theta ^{(1)} = \frac{1}{(q-1)t}( \sigma - \Id \sb M).
\]
\end{enumerate}
holds. 
Under these conditions,  
 their Picard-Vessiot extension is 
realized in the category of commutative algebras.
The second assumption seems too restrictive as clearly explained in \cite{masy}.  
If we drop one of these conditions, there are many linear 
\itqsi equations whose Picard-Vessiot ring is not commutative and Galois group is a quantum group
that is neither commutative nor co-commutative.  
\par 
We analyze only one favorite example \eqref{176} over the base field $\com$,  
 which  is equivalent to the non-linear equation 
in \cite{saiume13.2}.  The reader's imagination would go far away. 
In the example, we have a Picard-Vessiot ring
$R$ that is non-commutative, simple \itqsi ring  
(Observation \ref{obs3} and Lemma \ref{1016a}).  
The Picard-Vessiot ring $R$ is a torsor of a quantum group (Observation \ref{obs5}).  
 We have the Galois correspondence (Observation \ref{1018c}) and non-commutative Tannaka theory (Observation \ref{1018b}). 
 \par 
 We are grateful to A.~Masuoka and K.~Amano 
for teaching us their Picard-Vessiot theory 
and clarifying our idea.   
\section{Field extension $\com (t) /\com $ from classical and quantum view points}  
In \S 8 of the previous paper \cite{saiume13.2}, we studied a non-linear 
\qsi equation, which we call   
{\it qsi} equation for short, 
\begin{equation}\label{171}
\theta^{(1)}(y) =1,\qquad \sigma(y) =qy, 
\end{equation}
where  $q \not= 0,\, 1$ is a complex number. 
Let $t$ be a variable over the complex number field $\com$. 
We assume to simplify the situation that $q$ is not a root of unity. 
We denote by $\sigma \colon \com(t) \to \com(t)$ the $\com$-automorphism of the field $\com(t)$ 
of rational functions 
sending $t$ to $qt$. 
We introduce the 
$\com$-linear operator  $\theta^{(1)} \colon \com(t) \to \com(t)$
 by 
\[
\theta^{(1)}\left( f(t)\right) := \frac{f(qt)-f(t)}{(q-1)t} \quad\text{ for every  $f(t) \in \com(t)$. }
\]
We set
\[
\theta^{(m)} := 
\begin{cases}
\Id_{\com(t)}& \text{ for $m=0$}\\
\frac{1}{[m]_{q}!}\left(\theta^{(1)}\right)^{m} & \text{ for $m=1,\,2,\,\cdots$. }
\end{cases}
\]
As we assume that $q$ is not a root of unity, 
the number $[m]\sb q$ 
in the formula is  
not equal to $0$ and hence   
the formula determines the family $\theta ^{*}=
\{ \theta ^ {(i)}\,  | \, i \in \N \, \}$ of operators. 
So $(\com(t),\, \sigma,\, \theta^{\ast})$ is a {\it qsi} field. See \S 7, \cite{saiume13.2} and  $y=t$ is a solution for system \eqref{171}.

The system \eqref{171} is non-linear in the sense that for two solutions $y_{1},\, y_{2}$ of 
\eqref{171}, a $\com$-linear combination $c_{1}y_{1}+c_{2}y_{2}\,(c_{1},\,c_{2} \in \com)$ is not a solution of the system. 

However, the system is very close to a linear system. 
To illustrate this, let us look at the differential field extension $(\com (t), \partial \sb t )/ (\com , \, \partial \sb t )$, where we denote the derivation 
$d/dt$ by $\partial_{t}$   
The variable $t \in \com(t)$ satisfies a non-linear differential equation
\begin{equation}\label{172}
\partial_{t} t -1 = 0.
\end{equation}
The differential field extension $(\com(t),\, \partial_{t})/(\com,\,\partial_{t})$ is, however, the Picard-Vessiot extension for the linear differential equation
\begin{equation}\label{173}
\partial^{2}_{t} t = 0.
\end{equation}
To understand the relation between \eqref{172} and \eqref{173}, we introduce the 
$2$-dimensional 
$\com$-vector space 
\[
E := \com t \oplus \com \subset \com[t].
\]
The vector space $E$ is close under the action of the 
 derivation $\partial \sb t$ so that $E$ is a 
 $\com [\partial\sb t]$-module.
Solving the differential equation associated with the $\com[\partial_{t}]$-module $E$ is to find a differential algebra $(L,\, \partial_{t})/\com$ and 
a $\com[\partial_{t}]$-module morphism
\[
\varphi \colon E \rightarrow L. 
\]
Writing $\varphi(t) = f_{1},\, \varphi(1)=f_{2}$ that are elements of $L$, we have 
\[
\begin{bmatrix}
\partial_{t} f_{1}\\
\partial_{t} f_{2}
\end{bmatrix} = \begin{bmatrix}
0&1\\\
0&0
\end{bmatrix} \begin{bmatrix}
f_{1}\\
f_{2}
\end{bmatrix}.
\]
Since $\partial_{t} t = 1,\, \partial_{t} 1= 0$,  
in the differential field $(\com(t),\, \partial_{t})/\com$, we find two 
solutions
$
^ t (t , \,  1)
$ and $
^t (1, \, 0 )
$ 
that are two column vectors in $\com (t)^2$ satisfying  
\begin{equation}\label{175}
\partial_{t} \begin{bmatrix}
t & 1\\
1 & 0
\end{bmatrix}=\begin{bmatrix}
0 & 1\\
0 & 0
\end{bmatrix}
\begin{bmatrix}
t & 1\\
1 & 0
\end{bmatrix}
\end{equation} 
and  
$$
\begin{vmatrix}
t & 1\\
1 & 0
\end{vmatrix}
\not= 0.
$$
Namely, $\com(t)/\com$ is the Picard-Vessiot extension for linear differential equation \eqref{175}. 
\par
We can argue similarly for the {\it qsi} field extension $(\com(t),\, \sigma,\, \theta^{\ast})/\com$. 
You will find a subtle difference between the differential case and the
 {\it qsi} case. 
Quantization of Galois group arises from here.
\par 
Let us set 
$$
M = \com t \oplus \com \subset \com[t]
$$
 that is a $\com[\sigma,\, \theta^{\ast}
 ]$-module. 
Maybe to avoid the confusion that you might have in Remark \ref{mas} below, 
writing $m\sb 1 = t$ and  $m\sb 2 =1$, 
  we had better define formally 
  $$
  M = \com m\sb 1 \oplus  \com m\sb 2
  $$ 
  as a $\com$-vector space on which $\sigma$ and $\theta ^{(1)}$ operate by 
\begin{equation}  
\begin{bmatrix}
\sigma (m_{1}) \\ 
\sigma (m_{2})
\end{bmatrix} = \begin{bmatrix}
q & 0\\
0 & 1
\end{bmatrix}\begin{bmatrix}
 m_{1}\\
 m_{2}
\end{bmatrix},  \qquad 
                  \label{176}
\begin{bmatrix}
\theta^{(1)}( m_{1})\\
\theta^{(1)} (m_{2})
\end{bmatrix} = \begin{bmatrix}
0 & 1\\
0 & 0
\end{bmatrix}\begin{bmatrix}
 m_{1}\\
 m_{2}
\end{bmatrix}. 
\end{equation}
Solving $\com[\sigma,\, \theta^{\ast}]$-module $M$ is equivalent to 
find elements $f\sb 1 , \, f\sb 2 $ in a {\it qsi} algebra 
$(A, \, \sigma , \, \theta ^ * )$ satisfying   
the  system of linear difference-differential equation
\begin{equation}
\begin{bmatrix}
\sigma (f_{1})\\
\sigma (f_{2})
\end{bmatrix} = \begin{bmatrix}
q & 0\\
0 & 1
\end{bmatrix}\begin{bmatrix}
 f_{1}\\
 f_{2}
\end{bmatrix}, \qquad \label{176b}
\begin{bmatrix}
\theta^{(1)}( f_{1})\\
\theta^{(1)} (f_{2})
\end{bmatrix} = \begin{bmatrix}
0 & 1\\
0 & 0
\end{bmatrix}\begin{bmatrix}
 f_{1}\\
 f_{2}
\end{bmatrix} 
\end{equation}
in the \itqsi algebra $A$.

\begin{lemma} \label{lem1}
Let $(L,\, \sigma,\, \theta^{\ast})/\com$ be a {\it qsi} field extension. If a $2\times 2$ matrix $Y=(y_{ij}) \in \mathrm{M}_{2}(L)$ satisfies 
a system of difference-differential equations 
\begin{equation}
\sigma Y = \begin{bmatrix}
q & 0\\
0 & 1
\end{bmatrix}Y \textit{ and } 
\label{1720} 
\theta^{(1)}Y = \begin{bmatrix}
0 & 1\\
0 & 0
\end{bmatrix}Y,  
\end{equation}
then $\det Y = 0$. 
\end{lemma}
\begin{proof}
It follows from \eqref{1720} 
\begin{equation}\label{1722}
\sigma(y_{11}) = qy_{11},\; \sigma(y_{12})=qy_{12},\; \sigma(y_{21}) = y_{21},\; \sigma(y_{22}) = y_{22}
\end{equation}
and 
\begin{equation}\label{1723}
\theta^{(1)}(y_{11}) = y_{21},\, \theta^{(1)}(y_{12})= y_{22},\, \theta^{(1)}(y_{21}) = 0,\, \theta^{(1)}(y_{22}) = 0. 
\end{equation}
It follows from \eqref{1722} and \eqref{1723}
\begin{equation} \label{1724}
\theta^{(1)}(y_{11}y_{12}) = \theta^{(1)}(y_{11})y_{12} + \sigma(y_{11})\theta^{(1)}(y_{12}) = y_{21}y_{12} + qy_{11}y_{22}
\end{equation}
and similarly 
\begin{equation}\label{1725}
\theta^{(1)}(y_{12}y_{11}) = y_{22}y_{11} + qy_{12}y_{21}. 
\end{equation}
As $y_{11}y_{12}=y_{12}y_{11}$, equating \eqref{1724} and \eqref{1725}, we get
\[
(q-1)(y_{11}y_{22}-y_{12}y_{21}) = 0
\]
so that $\det Y = 0$. 
\end{proof}
\begin{cor}\label{cor2}
Let $(K,\, \sigma,\, \theta)$ be a {\it qsi} field over $\com$. Then the {\it qsi} linear equation
\begin{equation}\label{1011a}
\sigma Y = \begin{bmatrix}
q & 0\\
0 & 1
\end{bmatrix}Y \, \text{ and }\, 
\theta^{(1)}Y = \begin{bmatrix}
0 & 1\\
0 & 0
\end{bmatrix}Y
\end{equation}
has no {\it qsi} Picard-Vessiot extension
$L/K$ in 
 the following sense. There exists a solution 
$Y \in \GL \sb 2 (L)$ to \eqref{1011a} such that the abstract field $L$ is generated by the entries of the matrix $Y$ over $K$. The field of constants of the {\it qsi} over-field $L$ coincides with 
the  field of constants of the base 
field $K$. 
 \end{cor}
\begin{proof}
This is a consequence of Lemma \ref{lem1}.
\end{proof}
\begin{remark}\label{mas}
Masuoka pointed out that 
Corollary \ref{cor2} is compatible with Remark 4.4 and Theorem 4.7 of Hardouin \cite{har10}. See also Masuoka and Yanagawa \cite{masy}.  
They assure  
the existence of Picard-Vessiot extension for a   
$K[ \sigma , \, \theta ^ *]$-module $N$ if the following two conditions are satisfied;  
\begin{enumerate}
 \item 
 The {\it qsi} base field  $K$ contains $(\com (t), \, \sigma , \, \theta^* )$, 
 \item
  The operation of $\sigma$ and  $\theta^{(1)} $ on the module $N$ 
  as well as on the base field $Y$,  
  satisfy the relation 
$$
\theta ^{(1)} = \frac{1}{(q-1) t}(\sigma - \Id \sb N).
$$ 
\end{enumerate}
In fact, even if the base field $K$ contains 
$(\com (t), \, \sigma , \, \theta ^ * )$, in 
$K\otimes \sb \com   M$, we have by definition of 
the 
$\com [\sigma , \, \theta ^*]$-module $M$, 
$$
\theta ^{(1)} (m\sb 1 ) = m\sb 2 \not= 
\frac{1}{t}m\sb 1 = 
\frac{1}{(q-1)t}
\left( \sigma (m\sb 1) - m\sb 1\right) .
$$ 
So 
$K\otimes \sb \com   M$
does not satisfy the second condition above.

\end{remark}
\section{Quantum normalization of $( \com (t) , \sigma , \,
 \theta ^{*})/ \com $}
We started from the {\it qsi} field extension $\com(t)/\com$. The  column vector 
$^t(t,\, 1) \in \com(t)^{2}$ is a solution to the system of equations \eqref{176}, i.e. we have
\[
\begin{bmatrix}
\sigma (t)\\
\sigma (1)
\end{bmatrix} = \begin{bmatrix}
q & 0\\
0 & 1
\end{bmatrix}\begin{bmatrix}
t\\
1
\end{bmatrix}, \\
\qquad
\begin{bmatrix}
\theta^{(1)} (t)\\
\theta^{(1)} (1)
\end{bmatrix} = \begin{bmatrix}
0 & 1\\
0 & 0
\end{bmatrix}\begin{bmatrix}
t\\
1
\end{bmatrix} .
\]
By 
applying to the {\it qsi} field extension $( \com (t) , \sigma , \,
 \theta ^{*})/ \com $, 
the general procedure of 
\cite{ume96.2}, 
\cite{hei10} that is believed to lead us to the normalization,
we arrived at the Galois hull 
 $\eL = \com(t) [Q,\, Q^{-1},\, X]$. 
This suggests an appropriate normalization of the non-commutative 
\itqsi  ring extension $\com(t)[Q,\, Q^{-1}]/\com$ is a (maybe the), {\it qsi} Picard-Vessiot extension of  the system of equations \eqref{176}. More precisely, $Q$ is a variable over $\com(t)$ satisfying the commutation relation
\[
Qt=qtQ. 
\]
We understand $R=\com[t,\, Q,\,Q^{-1}]$ as a sub-ring of $S=\com[[t,\, Q]][t^{-1},\,Q^{-1}]$. 
The ring $S$ is a non-commutative {\it qsi} algebra by setting 
\[
\sigma(Q) = qQ,\, \theta^{(1)}(Q) = 0\, \textit{ and }\, \sigma(t) = qt,\, \theta^{(1)}(t) = 1
\]
and $R=\com[t,\, Q,\,Q^{-1}]$ is a {\it qsi} sub-algebra. 
Thus we get a {\it qsi} ring extension $(R,\, \sigma,\, \theta^{\ast})/\com = (\com[t,\, Q,\, Q^{-1}],\, \sigma,\, \theta^{\ast})/\com$. We examine that $(R,\, \sigma,\, \theta^{\ast})/\com$ is a non-commutative Picard-Vessiot extension for the systems of equations \eqref{176}. 
\begin{observation}\label{obs1}
The $\com[\sigma,\, \theta^{\ast}]$-module $M$ has two solutions linearly independent over $\com$. In fact, setting
\begin{equation}
Y:= \begin{bmatrix}
Q & t \\
0 & 1
\end{bmatrix} \in \mathrm{M}_{2}(R), 
\end{equation}
we have 
\begin{equation}
\sigma Y = \begin{bmatrix}
q & 0\\
0 & 1
\end{bmatrix}Y \textit{ and }
\theta^{(1)}Y = \begin{bmatrix}
0 & 1\\
0 & 0
\end{bmatrix}Y. 
\end{equation}
So the column vectors $^{t}(Q,\, 0),\, ^{t}(t,\, 1) \in R^{2}$ are $\com$-linearly independent solution of the system of equations \eqref{176}. 
\end{observation}
\begin{observation}\label{obs2}
The ring $\com[t,\, Q,\, Q^{-1}]$ has no zero-divisors. We can consider the ring $K$ of total fractions of $\com[t,\, Q,\, Q^{-1}]$. 
\end{observation}
\begin{proof}
In fact, we have $R \subset\com[[t,\, Q]][t^{-1},\,Q^{-1}]$. In the latter ring every non-zero element is invertible. 
\end{proof}
\begin{observation}\label{obs3}
The ring of {\it qsi} constants $C_{K}$ coincide with $\com$. The ring of $\theta^{\ast}$ constants of $\com[[t,\, Q]][t^{-1},\,Q^{-1}]$ is $\com (Q)$. Moreover 
as we assume that $q$ is not a root of unity, 
the ring of $\sigma$-constants of $\com(Q)$ is equal to $\com$. 
\end{observation}
\begin{lemma}\label{1016a}
The non-commutative \itqsi algebra $R$ is simple. There is no  \itqsi bilateral ideal of $R$ except for the zero-ideal and $R$.
\end{lemma}  
\begin{proof}
Let $I$ be a non-zero \itqsi bilateral ideal of $R$. 
We take an element 
$$
 0 \not= f:= a\sb 0 + ta\sb 1 + \ldots + t^n\,  a\sb n \in I,
$$
where $a\sb i \in \com [ Q, \, Q^{-1}]$ for $ 0 \le i \le n$. We may assume $a\sb n \not= 0$. 
Applying $\theta^{(n)}$ to the element $f$, we conclude that 
$ 0 \not= a\sb n \in \com [Q, \, Q^{-1} ]$ is in the ideal $I$.  Multiplying a monomial $bQ^l$ with $b\in \com $, we find 
a polynomial 
$h = 1 + b\sb 1 Q + \ldots + b\sb s Q ^s \in \com [Q]$ 
with $b\sb s \not= 0$ 
is in the ideal $I$. 
We show that $1$ is in $I$ by induction on $s$.
If $s\, =\, 0$, then there is nothing to prove. 
Assume that the assertion is proved for $ s \le m$. We have to show the assertion for $s = m+1$. 
Then,
since $Q^i$ is an eigenvector of the operator 
$\sigma$ with 
eigenvalue $q^i$ for $i \in \N$, 
$$
\frac{1}{q^{m+1} -1}(q^{m+1}h \, - \, \sigma (h)) 
= 1 + c\sb 1 Q + \ldots + c\sb {m}Q^m \in \com [Q]
$$
is an element of $I$ and by induction hypothesis 
$1$ is in the ideal $I$. 
\end{proof}  
\begin{observation}\label{obs4}
The extension $R/\com$ trivializes the $\com[\sigma,\, \theta^{\ast}]$-module $M$. Namely, there exist constants $c_{1},\,c_{2} \in R \otimes_{\com}M$ such that
\[
 R \otimes_{\com}M \simeq Rc_{1}\oplus Rc_{2}. 
\]
\end{observation}
\begin{proof}
In fact, it is sufficient to set
\[
c_{1} := Q^{-1}m_{1}-Q^{-1}tm_{2},\,\qquad  c_{2} := m_{2}. 
\]
Then 
\[
\sigma(c_{1}) = c_{1},\, \qquad \sigma(c_{2}) = c_{2},\, \qquad \theta^{(1)}(c_{2}) = 0
\]
and
\[
\theta^{(1)}(c_{1}) = q^{-1}Q^{-1}\theta^{(1)}(m_{1})
-q^{-1}Q^{-1}m_{2}= q^{-1}Q^{-1}m_{2}
-q^{-1}Q^{-1}m_{2} = 0. 
\]
So we have an $(R,\, \sigma,\, \theta^{\ast})$-module isomorphism $ R \otimes_{\com}M \simeq Rc_{1}\oplus Rc_{2}$. 
\end{proof}
\begin{observation}\label{obs5}
The Hopf algebra $\mathfrak{h}_{q}= \com \langle u,\, u^{-1},\, v\rangle$ with $uv=q\,vu$ co-acts on the non-commutative algebra 
$R$. Namely, we have an algebra morphism 
\[
R \to 
R \otimes \sb \com \mathfrak{h}\sb q
\]
sending 
\[
t \mapsto t\otimes 1 \, +\,  Q \otimes  v,\qquad Q\mapsto Q \otimes u, \qquad Q^{-1} \mapsto Q^{-1}\otimes u^{-1}.
\] 
This gives us an $1\otimes R$-algebra isomorphism
\begin{equation}\label{241}
 R\otimes \sb \com R \to 
 R\otimes \sb \com 
 \mathfrak{h}_{q} 
\end{equation}
Moreover isomorphism \eqref{241} is $\com[\sigma,\, \theta^{\ast}]$-isomorphism, where $\com[\sigma,\, \theta^{\ast}]$ operates on the Hopf algebra $\mathfrak{h}_{q}$ trivially. 
\end{observation}
\par 
We study category $\mathcal{C}(\com[\sigma,\, \theta^{\ast}])$ of left $\com[\sigma,\, \theta^{\ast}]$-modules that are finite dimensional as $\com$-vector spaces. 
We notice first the internal homomorphism 
\[
\mathrm{Hom}_{\com}((M_{1},\, \sigma_{1},\, \theta^{\ast}_{1}),(M_{2},\, \sigma_{2},\, \theta^{\ast}_{2})) \in ob(\mathcal{C}(\com[\sigma,\, \theta]))
\]
exists for two objects $(M_{1},\, \sigma_{1},\, \theta^{\ast}_{1}),(M_{2},\, \sigma_{2},\, \theta^{\ast}_{2}) \in ob(\mathcal{C}(\com[\sigma,\, \theta]))$. 
In fact, let $N :=\mathrm{Hom}(M_{1}, M_{2})$ be the set of $\com$-linear maps from $M_{1}$ to $M_{2}$. 
It sufficient to consider two $\com$-linear maps
\[
\sigma_{h} \colon N \rightarrow N \text{ and } \theta^{(1)}_{h} \colon N \rightarrow  N 
\]
given by
\[
\sigma_{h}(f) := \sigma_{2} \circ f \circ \sigma_{1}^{-1} \text{ and } \theta^{(1)}_{h}(f):= -(\sigma_{h}f)\circ\theta^{(1)}_{1} + \theta^{(1)}_{2}\circ f.
\]
So we have $q \sigma_{h}\circ \theta^{(1)}_{h} = \theta^{(1)}_{h} \circ \sigma$. 
Since $q$ is not a root of unity, we define $\theta^{(m)}_{h}$ in an evident manner
\[
\theta^{(m)}_{h} = \begin{cases}
\Id_{N}, \quad \text{ if } m=0, \\
\frac{1}{[m]_{q}!}\left(\theta^{(1)}_{h}\right)^{m}, \quad \text{ if } m \geq 1.
\end{cases}
\]
Since $\com[\sigma,\, \theta^{\ast}]$ is a Hopf algebra, for two objects $M_{1},\, M_{2} \in ob(\mathcal{C}(\com[\sigma,\, \theta]))$ the tensor product $M_{1}\otimes_{\com}M_{2}$ is defined as an object of $\mathcal{\com [\sigma,\, \theta^{\ast}]}$. 
However, as $\com[\sigma\, \theta^{\ast}]$ is not co-commutative, we do not have, in general, $M_{1}\otimes_{\com}M_{2} \simeq M_{2}\otimes_{\com}M_{1}$. 
Taking the forgetful functor
\[
\omega \colon \mathcal{C}(\com[\sigma, \theta^{\ast}]) \rightarrow \text{ categoly of $\com$-vector spaces, }
\]
we get a non-commutative Tannakian category. 
\begin{observation}\label{1018b}
The non-commutative Tannakian category $\{\{ M \}\}$ generates by the $\com[\sigma,\, \theta^{\ast}]$-module $M$ is equivalent to 
the category $\mathcal{C}\, (\mathfrak{h}\sb{q})$ of
 right $\mathfrak{h}_{q}$-co-modules that are finite dimensional as $\com$-vector space
\end{observation}
\begin{proof}
We owe this proof to Masuoka and Amano.
Our Picard-Vessiot ring $R$ is not commutative. However, 
by Observations \ref{obs3}, \ref{obs4} and Lemma \ref{1016a}, we can apply the arguments of the  classical differential Picqrd-Vessiot theory according  to Amano, Masuoka and Takeuchi
 \cite{amaetal09}, \cite{amam05}. 
 We first show that every $\com [\sigma , \, \theta ^* ]$-module 
  $N \in \{\{ M \}\}$ is trivialized over $R$. 
 Then,   
the functor 
\[
\phi \colon \{\{M\}\} \rightarrow \mathcal{C}(\mathfrak{h}_{q})
\]
is given by 
\[
\phi(N) = \text{ Constants of $\com[\sigma,\, \theta^{\ast}]$-module $R\otimes_{\com}N$ for $N \in ob(\{\{M\}\})$. }
\]
In fact,  
the Hopf algebra $\mathfrak{h}_{q}$ co-acts  on 
$R$ and so on 
the trivial $R$-module 
$R\otimes_{\com}N$ and  consequently on the vector space of constants of $R\otimes_{\com}N$. 
\end{proof}
\begin{observation}\label{1018c}
We have the Galois correspondence between the elements of the  two sets. 
\begin{enumerate}
\item The set of intermediate {\it qsi} division rings of $K/\com$: 
$$
\com , \; \com (t) , \; \com (Q), \; K  
$$
with inclusions 
\begin{equation}\label{1113a}
\com \subset \com (t) \subset K \textit{ and }
\com \subset \com (Q) \subset K.
\end{equation} 
\item The set of 
quotient 
$\com$-Hopf algebras of $\mathfrak{h}_{q}$:
\[ \mathfrak{h}_{q},\,\mathfrak{h}_{q}/I,\,
\mathfrak{h}_{q}/J, \; 
 \com \]
 with the sequences of the quotient 
 morphisms corresponding to 
 inclusions 
 \eqref{1113a}
 \begin{equation}
 \mathfrak{h}_{q} \to \mathfrak{h}_{q}/I \to \com 
 \textit{ and } 
 \mathfrak{h}_{q} \to \mathfrak{h}_{q}/J \to \com ,
 \end{equation}
where $I$ is the ideal of the Hopf algebra 
$\mathfrak{h}_{q}$ generated by $u - 1$ and 
similarly 
$J$ is the ideal of the Hopf algebra generated by $v$.   
\end{enumerate}
\end{observation}
The extensions 
$$
K/\com , \, K/ \com (Q), \, K/ \com (t) 
\textit{ and }  \com (Q) / \com 
$$
are \itqsi Picard-Vessiot extensions with Galois groups
\begin{align*}
&\mathrm{Gal}\, (K/\com ) \simeq \mathfrak{h}_{q}, \quad
\mathrm{Gal}\, (K/\com (Q)) \simeq 
\com [\G\sb{a \com} ], 
\\
&\mathrm{Gal}\, (K/\com (t)) \simeq 
\com [\G \sb{a \com} ], \,
\textit{ and } \,
\mathrm{Gal}\, (\com (Q)/\com ) \simeq 
\com [\G \sb{m \com }]. 
\end{align*}
Here we denote by $\com [G]$ the Hopf algebra of the 
coordinate ring of an affine group scheme $G$ over $\com$. 
\bibliographystyle{plain}
\bibliography{umemura2b131011}
\end{document}